\documentclass[12pt,letterpaper,titlepage]{amsart}
\usepackage{amsmath, amssymb, amsthm, amsfonts,amscd,xr}

\newcommand{\Z}{\mathbb{Z}}
\newcommand{\R}{\mathbb{R}}
\newcommand{\C}{\mathbb{C}}

\newcommand{\oline}{\overline}

\def\beq{\begin{equation}}
\def\eeq{\end{equation}}

\def\arr{\hbox to 20pt{\rightarrowfill}}

\def\e{\varepsilon}

\def\sl{\mathfrak {sl}}

\newenvironment{res}
               {\begin{equation}
\begin{minipage}{0.85\textwidth}}
               { \end{minipage}\end{equation} }
\def\ber{\begin{res} }
\def\eer{\end{res}}

\numberwithin{equation}{section}
\newtheorem{thm}{Theorem}[section]

\newtheorem{lem}[thm]{Lemma}

\newtheorem{cor}[thm]{Corollary}

\newtheorem{prop}[thm]{Proposition}

\newtheorem{rem}[thm]{Remark}

\makeatletter
\def\section{\@startsection {section}{1}{\z@}{3.5ex plus 1ex minus
    .2ex}{2.3ex plus .2ex}{\large\bf}}
    \def\subsection{\@startsection{subsection}{2}{\z@}{3.25ex plus 1ex minus
 .2ex}{1.5ex plus .2ex}{\bf}}
\makeatother
\def\pf{{\em Proof}.\, }

\def\e{\epsilon}

\def\af{\mathfrak{a}}
\def\gf{\mathfrak{g}}

\def\kf{\mathfrak{k}}

\def\pf{\mathfrak{p}}

\def\P{\mathbb{P}}

\def\Aut{\operatorname{Aut}}
\def\ad{\operatorname{ad}}
\def\diag{\operatorname{diag}}
\def\Gr{\operatorname {Gr}}
\def\im{\operatorname{im}}

\def\SO{\operatorname{SO}}

\def\tr{\operatorname{tr}}

\def\Sym{\operatorname{Sym}}
\def\Span{\operatorname{span}}

\def\O{\mathcal{O}}

\def\W{\mathcal{W}}
\def\X{\mathcal{X}}
\def\Y{\mathcal{Y}}
\def\Zc{\mathcal{Z}}

\renewcommand{\Im}{\mbox{\rm Im}\,}
\makeindex
\begin{document}
\title[Normal forms]
{Normal forms for real quadratic forms}
\author{Bernhard Kr\"otz}
\address{Max-Planck-Institut f\"ur Mathematik, Vivatsgasse 7,
D-53111 Bonn, Germany\\ email: kroetz@mpim-bonn.mpg.de}
\author{Henrik Schlichtkrull}
\address{Department of Mathematical Sciences\\ 
University of Copenhagen\\ Universitetsparken 5 \\ 
DK-2100 K\o{}benhavn, Denmark
\\ email: schlicht@math.ku.dk} 
\date{April 29, 2008}
\begin{abstract}{We investigate the 
non-diagonal normal forms of a quadratic form 
on $\R^n$, in particular for $n=3$. For this case
it is shown that the set of normal forms
is the closure of a 5-dimensional submanifold
in the 6-dimensional Grassmannian of 2-dimensional 
subspaces of $\R^5$.}\end{abstract}
\maketitle 

\maketitle
\section{Introduction}

According to the principal axes theorem
every real quadratic form in $n$ variables
allows an orthogonal diagonalization with normal form
$$A_1x_1^2+\dots+A_nx_n^2,$$
where $A_1,\dots,A_n\in\R$.
In this article we investigate (for the case $n=3$)
the existence of other normal forms. 

To be more precise, let $q_1,\dots,q_n$ be 
quadratic forms on $\R^n$. If for every 
quadratic form $q$ on $n$-dimensional Euclidean space
there exists an orthonormal basis
in which $q$ takes the form
$$q(x)=A_1q_1(x)+\dots+A_nq_n(x)$$ 
for some set of coefficients,
we say that this expression is a {\it normal form} of $q$.

Passing to matrices,
let us consider ${V}=\Sym(n,\R)$, the vector space of symmetric 
$n\times n$-matrices. On ${V}$ there is the natural 
action of the special orthogonal group $K:=SO(n,\R)$ by conjugation, say 
$$k\cdot X:= kXk^{-1} \qquad (k\in K, X\in {V})\, .$$ 
If ${D}$ denotes the space of diagonal matrices in 
${V}$, then the principal axes theorem asserts that
$${V}= K\cdot {D}:=\{k\cdot d\mid k\in K, d\in D\}.$$ 
Furthermore if $d_1, d_2\in {D}$,
then $K\cdot d_2=K\cdot d_1$ if and only if 
$d_2$ is obtained from $d_1$ by a permutation of coordinates (the set of 
eigenvalues is unique). 

The question we address is, {\it for which $n$-dimensional subspaces
${W}$ in ${V}$ is ${V}=K\cdot{W}$?}
It would be tempting to assert that
the unique property of ${D}$ (and its conjugates by $K$), 
which causes the principal axes theorem, 
is that it is {\it abelian}. 
However this is not correct, in fact there exist
non-abelian $n$-dimensional subspaces $W$ with
${V}=K\cdot W$ (see the theorem below).

There is some redundancy in the problem, namely the center of ${V}$
on which $K$ acts trivially. 
Let us remove that and define ${\pf}:={V}_{\tr =0}$ to be the 
space of zero-trace  elements in ${V}$. Likewise we set 
$\af:= {D}_{\tr=0}$.  
The principal axes theorem now reads as 
$$\pf = K\cdot \af\, .$$

\begin{thm} \label{th=1}Let  $\pf=\Sym(3,\R)_{\tr =0}$ and 
$K=\SO(3,\R)$. Define 
$$W:=\left\{ X_{\mu\lambda}:=\begin{pmatrix} \mu & 0 & \lambda \\ 0 & -\mu & 0\\ \lambda & 0 & 0
\end{pmatrix}\mid \mu,\lambda\in\R\right\}\, .$$ 
Then $K\cdot W=\pf$. 
\end{thm}

\begin{proof} A more general result will be established later. 
Here we can give a simple proof.

Let $A\in\pf$ be given, and
let $\nu_1\ge \nu_2 \ge \nu_3$ be its eigenvalues. 
Then $\nu_1+\nu_2+\nu_3=0$, and
hence $\nu_1 \ge 0 \ge \nu_3$. 
Let 
$$\mu=-\nu_2=\nu_1+\nu_3 \quad,\quad
\lambda=\sqrt{-\nu_1\nu_3}.$$
The matrix $X_{\mu\lambda}$ has the characteristic polynomial 
$$
\begin{aligned}
\det \begin{pmatrix} \mu-x&0&\lambda\\0&-\mu-x&0\\\lambda&0&-x\end{pmatrix}
&=(-\mu-x)(-(\mu-x)x-\lambda^2)
\\&=(\nu_2-x)(x-\nu_1)(x-\nu_3)
\end{aligned}
$$
Hence $A$ and $X_{\mu\lambda}$ have the same eigenvalues, and thus they are conjugate.
\end{proof}

\begin{cor} Every trace free real quadratic form in three variables allows
a normal form of the type 
$$A(x^2-y^2)+ B xz$$
for $A, B\in \R$. 
\end{cor}

\par Let us more generally consider a real semi-simple Lie algebra 
$\gf$ with Cartan decomposition $\gf =\kf +\pf$. The space
$\pf=\Sym(n,\R)_{\tr =0}$ is obtained in the special case
$\gf=\sl(n,\R)$.
Let $\af\subset \pf$ be a maximal abelian subspace and $K=e^{\ad \kf}$. 
According to standard structure theory 
of semi-simple Lie algebras the following generalization of the
principal axes theorem holds:
\begin{itemize}
\item $\pf = K\cdot \af$. 
\item $K\cdot X = K\cdot Y$ for $X,Y\in \af$ if and only if $\W\cdot X=\W\cdot Y$ where 
$\W=N_K(\af)/ Z_K(\af)$ is the Weyl group. 
\end{itemize}

\par Let $r:=\dim \af$ be the real rank of $\gf$. We consider 
$\Gr_r(\pf)$ the Grassmannian of $r$-dimensional subspaces in $\pf$. 
Inside of $\Gr_r(\pf)$ we consider the subset 
$$\X:=\{ W\in \Gr_r(\pf)\mid K\cdot W =\pf\}\, .$$   
Then the following are immediate: 
\begin{itemize}
\item $\X=\Gr_r(\pf)$ if $r=1$. 
\item $\X\supset \X_{\rm ab}:=\{ W \in\Gr_r(\pf) \mid W \
\hbox{abelian}\}\simeq K/N$
where $N=N_K(\af)$. 
\end{itemize}
If $r\geq 2$ and $\gf$ simple, then $\X\subsetneq \Gr_r(\pf)$. 
The problem we pose is to determine $\X$ in general. 

In this paper we describe the set
$\X$ for $\gf=\sl(3,\R)$, in which case $r=2$
and $\dim\Gr_2(\pf)=6$. 
It turns out that $\X$ is 
dominated by a real algebraic variety of dimension 5:  
there exists a surjective algebraic map: 
$$\Phi: K\times_{N_0} \P(\R^3)\to \X$$
with $N_0\simeq (\Z/4\Z)\rtimes \Z/2\Z$ and generically trivial 
fibers (see Theorem \ref{th=2} in Section 3 below).  

\par In Section 4 we give an alternative approach to 
the problem of characterizing $\X$ via tools from algebraic geometry, 
in particular Galois-cohomology. 
This section evolved out of several discussions with G\"unter Harder and 
we thank him for explaining us some mathematics which was unfamiliar 
to us. 

\par For general $\gf$ we  do not know the nature of $\X$.

\section{Description by invariants}
Let $\gf=\sl(3,\R)$ and $\af=\diag(3,\R)_{\tr
  =0}$. We give the following
description of $\X$, which will lead to the classification
in the following sections.

\begin{thm}\label{eigenvalue theorem}
The two dimensional subspace $W\in \Gr_2(\pf)$
belongs to $\X$ if and only if it contains a non-zero
matrix $X$ with two equal eigenvalues.
\end{thm}

For example, with the notation
in Theorem \ref{th=1}, the matrix 
$X_{\mu\lambda}$ with $\mu=-1$ and $\lambda=\sqrt2$ has
eigenvalues $1,1,-2$. Hence the space $W$ in this theorem
belongs to $\X$.

\begin{proof} 
That this is a necessary condition is clear, since $W\in\X$
means that $W$ contains at least one element from {\it every} 
$K$-orbit
on $\pf$.

In order to describe the $K$-orbits,
we recall some basic invariant theory. 
Let
\begin{align*} u_1(X) & =\tr X^2 \\ 
u_2(X) &= \det X \end{align*}
for $X\in\pf$. Then $u_1,u_2\in \C[\pf]^K$, the ring 
of $K$-invariant polynomials on $\pf$. 
In fact, it is a well-known fact that
$$\C[\pf]^K =\C[u_1, u_2],$$
but we shall not use this here.

\begin{lem}\label{K-orbits} 
The level sets for $u=(u_1,u_2)$ are single $K$-orbits. 
\end{lem} 

\begin{proof} Each $K$-orbit is uniquely determined by a set of
eigenvalues (with multiplicities). It is easily seen that the
characteristic polynomial of a trace free $3\times 3$-matrix $X$ is
$$-x^3+\frac12 u_1(X)x + u_2(X).$$
The lemma follows immediately.
\end{proof}

It follows that $W$ belongs to $\X$ if and only if it
has a non-trivial intersection with each level set.
Notice that  $u_1(X)$ is the square
of the trace norm of $X$, for $X$ symmetric. 
In particular, $u_1(X)>0$ for
$X\neq 0$.
Since $u_1$ and $u_2$ are homogeneous, it suffices to consider
level sets of the form $\{ u_1=1, u_2=c_2\}$.

We thus consider for each $W\in \Gr_2(\pf)$
the unit sphere
$$W_1=\{X\in W\mid u_1(X)=1\},$$
and we define
$$J:=\{u_2(X)\mid X\in W_1\}.$$
Since $W_1$ is connected, $J$  is an interval.
Moreover, it is symmetric around $0$, since $u_2$ has
odd degree.
In particular, we denote by
$$ I:=\{ u_2(X) \mid X\in \af_1\}$$
the interval corresponding to the unit sphere
in $\af$.
We now show: 

\begin{lem} The interval $I$ is given by
$I =[-c,c]$, where $c={54}^{-1/2}$. Furthermore,
the extreme values $\pm c$
are obtained precisely in those elements $X\in\af_1$, 
which have two equal eigenvalues. 
\end{lem}

\begin{proof} 
Let us introduce coordinates for $\af$, namely 
$$\af=\{D_{xy}:= \diag(x,y, -x-y)\mid x,y\in\R\}\, .$$
Furthermore, let us introduce two functions:
\begin{align*} f_1(x,y) & :=u_1(D_{xy})= 2 (x^2 + y^2 +xy)\\ 
f_2(x,y) & := u_2(D_{xy})= -xy(x+y)\, .\end{align*}
We wish to maximize/minimize  $f_2$ under the condition of 
$f_1=1$. For that we perform the method of Lagrange:  
$df_1 = 2(2x+y, 2y+x)$  and $df_2= -(y(2x+y), x(2y+x))$ have to be 
collinear. This can only happen in three cases: $x=y$, $2x+y=0$ or
$2y+x=0$. Notice that these are exactly the cases in which
two of the diagonal entries of $D_{xy}$ are equal.

We start with $x=y$. Here $f_1(x,x)=6x^2=1$ means that $x=\pm 6^{-1/2}$. 
Hence $f_2(x,x)= -2x^3 =  \pm 54^{-1/2}$. 
Secondly, if $2x+y=0$, 
then $f_1(x,-2x)= 6x^2=1$, so again $x= \pm 6^{-1/2}$. 
Hence $f_2(x, -2x)=-2x^3= \pm 54^{-1/2}$. Finally, the case
$2y+x=0$ is similar. 
\end{proof}

In order to complete the proof of Theorem \ref{eigenvalue theorem},
we only have to note that, as 
$\pf =K\cdot \af$ we have $J\subset I$ and 
equality $J=I$ holds if and only if $W\in\X$. 
\end{proof}

\begin{rem}\label{remark} 
(a) It follows from Theorem \ref{eigenvalue theorem} that
not all 2-dimensional subspaces $W\subset\pf$
belong to $\X$. 
An extreme case is
$$W=\left\{ \begin{pmatrix} 0 & 0 & 0 \\ 0 & \lambda &\mu 
\\ 0 & \mu & -\lambda
\end{pmatrix}\mid \lambda, \mu \in \R\right\}\, .$$
for which $u_2(W)=\{0\}$ and hence $W\not \in \X$. 
\par\noindent (b) Let us define a continuous function on $\Gr_2(\pf)$ by 
$$f: \Gr_2(\pf)\to \R_{\geq 0} ,  \ \ W\mapsto \max_{X\in W_1} u_2(X)$$ 
Then we get 
$$\X=\{ W\in \Gr_2(\pf)\mid f(W)= 54^{-1/2}\}$$
by our previous result. In particular, $\X$ is a closed subset of
$\Gr_2(\pf)$. 
\end{rem}

\section{$K$-orbits on $\X$}

We aim to describe $\X$ explicitly.
Our starting point is the following observation.
Let
$$X_0:=\begin{pmatrix} 1 & 0 & 0  \\ 0  & 1 & 0 \\ 0 & 0 & -2 
\end{pmatrix}\, ,$$
and observe that the two dimensional subspace 
$W\in \Gr_2(\pf)$
belongs to $\X$ if and only if it contains a vector 
in the $K$-orbit of $X_0$.

In fact this is an immediate consequence 
of our discussion in the previous section.  
Since we consider trace free matrices,
$X$ has two equal eigenvalues if and only if its
eigenvalues are $\nu,\nu,-2\nu$ for some $\nu\in\R$,
that is, $X$ is conjugate to $\nu X_0$.

Let $\Omega$ denote the 3-dimensional variety
$$\Omega=\{W\in\Gr_2(\pf)\mid X_0\in W\},$$
in $\Gr_2(\pf)$. The stabilizer $H\subset K$ of the line 
$\R X_0$ acts on $\Omega$. 
We have proved the following result: 

\begin{prop}\label{prop surjective} 
The map $(k,W)\mapsto k\cdot W$ from
$K\times_H\Omega$ to $\X$ is surjective. 
\end{prop}

Notice that
$H=MT$ where $T$ is the maximal torus 
$$T:=\begin{pmatrix} \SO(2,\R) & 0 \\ 0 & 1\end{pmatrix}$$
and $M\simeq [\Z/ 2Z]^2$ is the diagonal group group generated by 
$$m_1:=\begin{pmatrix}
  -1  & 0   &0   \\  0 & -1   &  0  \\ 
0 & 0  & 1 \end{pmatrix},\qquad
m_2:=\begin{pmatrix}
  1  & 0   &0   \\  0 & -1   &  0  \\ 
0 & 0  & -1 \end{pmatrix}\, .$$
in $K$. 

For $Y\not\in \R X_0$  we set 
$$W_Y:=\Span_\R\{X_0, Y\}\, ,$$
then $\Omega=\{ W_Y \mid Y\not\in \R X_0\}$.
Since $X_0$ is fixed under the maximal torus $T$, we have 
that $W_Y$ and $W_{t\cdot Y}$ belong to the same 
$K$-orbit for $t\in T$. Thus it suffices to consider
elements $Y$ of the following shape: 
$$Y=Y_{\alpha, \delta,\e} =\begin{pmatrix} \alpha & 0 &\delta \\ 0 & -\alpha &\e \\ \delta& \e & 0 \end{pmatrix}.$$
Let $\Y$ denote the 2-dimensional
projective space of these lines and consider the algebraic mapping 
\begin{equation}\label{map}
K\times \Y \to \X, \ \ (k, [Y])\mapsto k\cdot W_Y\, .
\end{equation}

The group $H$ does not act on $\Y$. However,
let $N_0$ denote the subgroup of order 8,
generated by 
$$
s_0=\begin{pmatrix} 0 & 1 & 0 \\  -1 & 0 & 0 \\ 0 & 0 & 1\end{pmatrix}\, .
$$
and $m_2$ (note that $m_1=s_0^2$).
It follows from the relations 
\begin{equation}\label{N0 action}
s_0\cdot Y_{\alpha, \delta,\e}=Y_{-\alpha, -\e,\delta}\,,\quad
m_2\cdot Y_{\alpha, \delta,\e}=Y_{\alpha, -\delta, \e}\,.
\end{equation}
that $N_0$ acts on $\Y$. Conversely, if $k\in H$ and $k\cdot Y\in\Y$
for some $Y=Y_{\alpha, \delta,\e}\in\Y$ with $\alpha\neq 0$, then
$k\in N_0$.

Since furthermore $k\cdot W_Y=W_{k\cdot Y}$
for $Y\in\Y$ and $k\in N_0$,
the above map (\ref{map}) factorizes to an algebraic map
$$K\times_{N_0} \Y \to \X\, .$$
This map is $K$-equivariant, continuous and onto and   
we wish to show that it
is generically injective.

We  define the following open dense subset of $\Y$: 
\begin{equation} \Y'=\{ [Y_{\alpha, \delta, \e}]\mid 
\alpha\neq 0,\, \delta\neq0,\, \e\neq 0\} 
\end{equation}
and note that it is preserved by $N_0$.
We set $\X'=K\cdot \Y'\subset \X$.

\begin{thm}\label{th=2}
The map 
$$K\times_{N_0} \Y' \to  \X', \ \ [k,[Y]]\mapsto k\cdot W_Y$$
is a $K$-equivariant continuous bijection. 
In particular $\X'$ carries a natural structure of a 
smooth 5-dimensional $K$-manifold. 
\end{thm}

In order to obtain this, we study the intersection of the
$K$-orbit of $X_0$ with $W_Y$. 
We first prove:

\begin{lem}\label{lemma X}
Assume
$$X:=\begin{pmatrix} \lambda & 0   &\delta \\ 0 & \mu & \e \\ 
\delta& \e & -(\lambda + \mu )\end{pmatrix}\in K\cdot X_0.$$ 
Then $\delta=0$ or $\e=0$.
\end{lem} 

\begin{proof} 
It follows from Lemma \ref{K-orbits} that
\begin{align} 
u_1(X) &= 2(\lambda^2 +\mu^2 +\lambda \mu +\delta^2 +\e^2)=u_1(X_0)=6 \label{u1 equation}\\
u_2(X) &= -\lambda\mu (\lambda +\mu) -\e^2\lambda 
- \delta^2 \mu= u_2(X_0)=-2\, .\label{u2 equation}
\end{align} 
In particular, it follows from (\ref{u1 equation})
that
$\lambda^2 +\mu^2 +\lambda \mu \leq 3$.
Since $\lambda^2 +\mu^2 +\lambda \mu=
(\lambda+\frac12\mu)^2+\frac34\mu^2$ this implies
$|\mu| \leq 2.$

Multiplying  by $\frac12\mu$ in (\ref{u1 equation}) and adding
(\ref{u2 equation}) we obtain
$$\mu^3 +\e^2\mu-\e^2\lambda=3\mu-2 \,,$$
or equivalently
\begin{equation}\label{first eq}
\e^2(\lambda-\mu) =(\mu+2)(\mu-1)^2\, .
\end{equation}
In particular it follows that $\e=0$ or 
$\lambda\geq \mu$. 

Since $\mu$ and $\lambda$ appear symmetrically in
(\ref{u1 equation}) and (\ref{u2 equation}) we obtain
similarly $|\lambda|\leq 2$,
\begin{equation}\label{second eq}
\delta^2(\mu-\lambda) = (\lambda+2)(\lambda-1)^2\, ,
\end{equation}
and conclude that $\delta=0$ or 
$\mu\geq \lambda$. 

Notice finally that if $\lambda=\mu$, then  $\lambda=\mu=-2$ or 
$\lambda=\mu=1$ by 
(\ref{first eq}), and from (\ref{u1 equation})
it then follows that $\lambda=\mu=1$ and $\delta=\e=0$.
\end{proof}

We can now prove Theorem \ref{th=2}.

\begin{proof} It remains to be seen that $k\cdot W_Y=W_{Y'}$
implies $k\in N_0$ and $[Y']=[k\cdot Y]$ for $Y,Y'\in\Y'$.
In particular, it follows from $k\cdot W_Y=W_{Y'}$ that 
$k\cdot X_0=aX_0+bY'$ for some $a,b\in\R$. Since $Y'\in\Y'$
it follows from Lemma \ref{lemma X} that $b=0$, and hence
$k\in H$. Now $k\cdot Y$ must be a multiple of $Y'$,
by orthogonality with $X_0$ with respect to the trace form
$\langle A,B\rangle =\tr(AB)$. It follows that $k\in N_0$.
\end{proof}

In order to give a complete classification of 
$\X$, one needs to describe the fibers in $K\times\Y$ above
the elements outside of $\X'$. We omit the details, but mention
that in general the fibers will not be finite.

\section{Alternative approach}

In the following two subsections We describe an alternative 
approach to the elements in $\X$, based on results from
algebraic geometry and possibly useful in the general case. 

\subsection{Generic subspaces} 

Let $L\in \Gr_2(\pf)$. The following property of $L$
is closely related to the property that $K\cdot L=\pf$.
We say that $L$ is {\it generic} if there
exists an element $Z\in L$ such that
\begin{equation}\label{def generic}
[\kf,Z]+L=\pf.
\end{equation}
By reason of dimension, the sum is necessarily direct
if (\ref{def generic}) holds.
Equivalent with  (\ref{def generic}) is that
the map $(k,W)\mapsto k\cdot W$ is submersive at $(1,Z)$.
It follows that $L$ is generic if and only if the
image $K\cdot L$ has non-empty interior in $\pf$. In
particular, if $L\in\X$, then $L$ is generic.

It follows from Proposition \ref{prop surjective}
that $W_Y:=\Span_\R\{X_0, Y\}$ is generic for every
$Y\in\Y$. As we want to proceed independently 
of the computations in Section 2,
we sketch a simple proof of this fact.
Let $Y=Y_{\alpha,\delta,\e}$ where
$(\alpha,\delta,\e)\neq(0,0,0)$, and put
$Z=Y+cX_0$ where $c\in\R$. We claim that
(\ref{def generic}) holds for some $c$.

Let $\kf=\Span_\R(X_1,X_2,X_3)$ where
$$X_1=
\begin{pmatrix}
 0 & 1  &  0  \\ 
-1 & 0  &  0  \\ 
 0 & 0  &  0 
\end{pmatrix}
,\qquad
X_2=
\begin{pmatrix}
 0 & 0  &  1  \\ 
 0 & 0  &  0  \\ 
-1 & 0  &  0 
\end{pmatrix}
,\qquad
X_3=
\begin{pmatrix}
 0 & 0  &  0  \\ 
 0 & 0  &  1  \\ 
 0 &-1  &  0 
\end{pmatrix}.
$$
Then
$$
\begin{aligned}
&[X_1,Z]=
\begin{pmatrix}
 0        & -2\alpha  &  \e  \\ 
-2\alpha & 0          & -\delta  \\ 
\e       & -\delta   &  0 
\end{pmatrix}
\\
&[X_2,Z]=
\begin{pmatrix}
2\delta    & \e       &  -3c-\alpha  \\ 
 \e        & 0         &  0  \\ 
-3c-\alpha & 0         & -2\delta
\end{pmatrix}
\\
&[X_3,Z]=
\begin{pmatrix}
 0      & \delta    &  0  \\ 
\delta & 2\e       & -3c+\alpha  \\ 
 0      &-3c+\alpha &  -2\e 
\end{pmatrix}
\end{aligned}
$$
Hence the condition that $[X_1,Z]$, $[X_2,Z]$, $[X_3,Z]$,
$X_0$ and $Y$ are linearly independent amounts to
$$
\det\begin{pmatrix}
0&-2\alpha&\e&0&-\delta\\
2\delta&\e&-3c-\alpha&0&0\\
0&\delta&0&2\e&-3c+\alpha\\
1&0&0&1&0\\
\alpha&0&\delta&-\alpha&\e
\end{pmatrix}\neq0.
$$
This determinant is a second order polynomial in $c$.
It is easily seen that the coefficient of $c^2$ is
$18\alpha^2$. On the other hand, if $\alpha=0$ then 
the constant term in the polynomial is $2(\delta^2+\e^2)^2$.
In any case, it is a non-zero polynomial in $c$,
and our claim is proved.

It is of interest also to 
see which other spaces $L$ are generic. 

\begin{lem} Let $L\in\Gr_2(\pf)$. Then $L$ is generic
if and only if it is conjugate to
$\Span_\R\{X, Y\}$, where
$$X:=\begin{pmatrix}\lambda&0&0\\0&\mu&0\\0&0&-\lambda-\mu
\end{pmatrix},\quad
Y:=\begin{pmatrix}\alpha&\gamma&\delta\\\gamma&\beta&\e\\
\delta&\e&-\alpha-\beta
\end{pmatrix}$$
and either

{\rm (i)} two of the elements $\lambda,\mu,-\lambda-\mu$ in $X$ are equal

\noindent or 

{\rm (ii)}  $\alpha\mu-\beta\lambda\neq 0$. 
\end{lem}

\begin{proof}
Before the proof we make the following observation. Let  $Z\in L$
where $L\in\Gr_2(\pf)$ is arbitrary.
It follows easily from the relation
$\tr([U,Z]V)=\tr(U[Z,V])$ for $U\in\kf$ and $Z,V\in\pf$,
that $[\kf,Z]$ can be characterized as the set of elements $T\in\pf$,
for which $\tr(TV)=0$ for all $V$ in the centralizer of $Z$ in $\pf$.

Assume now that $L=\Span_\R\{X, Y\}$ as above. In case (i),
$L$ belongs to $\X$ and
is generic as established above. Assume (ii) and not (i).
Since the diagonal elements of $X$ are mutually different
it follows that the
centralizer of $X$ in $\pf$ is $\af$, and hence
$[\kf,X]$ consists of the matrices in $\pf$
with zero diagonal entries. Since $\alpha\mu-\beta\lambda\neq 0$
it follows that $X$ and $Y$
are linearly independent from $[\kf,L]$. Hence
$L$ is generic.

Conversely, if $L\in\Gr_2(\pf)$ is generic, then by conjugation
we can arrange that the matrix $Z$ in (\ref{def generic}) is diagonal.
Let $X=Z$. As before it follows that if (i) does not hold, then   
$[\kf,Z]$ consists of all the matrices which are
zero on the diagonal. Hence any $Y\in L$ 
linearly independent from $Z$
must have the mentioned form.
\end{proof}

For example, the subspace $W$ in Remark \ref{remark} (a) is not generic.
On the other hand, it 
is not difficult to find examples of subspaces which are generic, but do
not belong to $\X$. For example when 
$\lambda=0$, $\mu=1$, $\alpha=1$, $\beta=0$,
$\gamma=\delta=0$ and $|\epsilon|>3/2$ in the expressions above,
then $L=\Span_\R\{X, Y\}$ is generic and not in $\X$.

\subsection{Approach via algebraic geometry}
The following evolved from discussions with G\"unter Harder. 

\par Let $\Gr_2(\pf_\C)$ be the complex variety of $2$-dimensional 
complex subspaces of $\pf_\C$. 
For a subspace $L\in \Gr_2(\pf_\C)$ we denote by 
$\oline L$ its complex conjugate. Note that there is 
a natural embedding $\Gr_2(\pf)\hookrightarrow \Gr_2(\pf_\C)$ 
the image of which,
$$\Gr_2(\pf)=\{ L\in \Gr_2(\pf_\C)\mid L=\oline L\}\, ,$$  
constitutes the real points
of $\Gr_2(\pf_\C)$. Let $L_\R=L\cap\pf\in \Gr_2(\pf)$ 
for $L\in \Gr_2(\pf_\C)$ with $\oline L=L$.

\par As before we call $L\in \Gr_2(\pf_\C)$ {\it generic} if 
there exist $Z\in L$ such that
$$ L+[\kf_\C,Z]=\pf_\C\,$$
or equivalently, such that
$$\Phi_L: K_\C\times L \to \pf_\C, \quad (k,z)\mapsto k\cdot z$$
is submersive at $(1,Z)$.
It is clear that the complexification of a
generic subspace in  $\pf$ is generic.
Let us remark that if $L$ is generic then
\begin{itemize}
\item The image $\im\Phi_L$ has non-empty 
Zariski open interior in $\pf_\C$. 
\item The orbit $\O_L:=K_\C\cdot L \in \Gr_2(\pf_\C)$
satisfies: 
$$\dim_\C \O_L= \dim K_\C = 3\, .$$
\end{itemize}

\par Let now $L$ be generic. To $\O_L$ we associate: 
$$\Zc_L:=\{(z,W)\mid z\in W, \ W\in \O_L\}\, .$$
The projection onto the second factor $\pi_2: \Zc_L\to \O_L$ reveals 
the structure of an algebraic  $\C^2$-vector bundle over 
$\O_L$. In particular $\dim_\C \Zc_L=5$.
The projection onto the first factor $\pi_1: \Zc_L \to \pf_\C$
features  
$$\im \pi_1= \im \Phi_L\, .$$
In particular $\im \pi_1$ contains a non-empty Zariski-open set. 

\par Now let us assume that $L$ is the complexification
of a generic subspace in $\pf$. 
Then $\O_L$ and $\Zc_L$ are defined over $\R$. 
The real points of $\Zc_L$ are given by 

$$\X_L:=\Zc_L^\R=\{ (x, W)\in \Zc_L 
\mid W=\oline W, \ x\in W_\R\}\, .$$
Again $\pi_1(\X_L)\subset \pf$ is a constructible set with 
non-empty open interior.

\par In order to determine $\X_L$ we have to determine 
the real points of $\O_L$. In general this is a finite union 
of $K$-orbits which is difficult to determine as one needs to know
the $K_\C$-stabilizer of $L$. 

\par However, if we suppose that $L$ is such that $\O_L\simeq K_\C$,
then 
the real points of $\O_L$ are just $K\cdot L$ and 
$\im \pi_1(\X_L)=K\cdot L_\R \subset \pf$. 
Hence $K\cdot L_\R$ has non-empty Zariski open interior and thus 
$K\cdot L_\R=\pf$ as the left hand side is closed. 
We have thus established that $L\in\X$ for
every generic $L\in\Gr_2(\pf)$ with trivial stabilizer in $K_\C$.

Notice that the stabilizer in $K_\C$ is trivial if the stabilizer
in $K$ is trivial. This can be seen as follows.
Let us denote by $S\subset K_\C$ the stabilizer of 
$L$. As $L$ is generic, $S$ is a discrete subgroup of $K_\C$. 
We have to show that $S$  is trivial if $S\cap K$ is trivial. 
Note that $S=\oline S$. 
Therefore, for $k \in S$, 
$$x:= \oline k^{-1} k \in S\, .$$
Observe that $x=\exp(X)$ for a unique 
$X\in i\kf$. As $x$ is positive definite,  it follows from 
$x\cdot L= L$ that $\exp(\R X)\subset S$. Thus $X=0$ by the 
discreteness of $S$, and hence $k=\oline k$.

Let $\Y''$ denote the following subset of $\Y'\subset\Y$
\begin{equation} \Y''=\{ [Y_{\alpha, \delta, \e}]\mid 
\alpha\neq 0,\, \delta\neq0,\, \e\neq 0,\, \delta\neq \pm\e\} 
\end{equation}
We claim that for $[Y] \in \Y''$, the $K$-stabilizer 
of $W_Y$ is trivial. Assume $k\cdot W_Y=W_Y$ for some $k\in K$.
In the proof of Theorem \ref{th=2} we saw that $k\in N_0$
and $k\cdot Y=\pm Y$,
and then it follows from  (\ref{N0 action}) that $k=e$.
 
\par To summarize, we have shown with alternative methods that 
$W_Y\in\X$ for $Y\in\Y''$.

\par Let us now deal with generic orbits $\O:=\O_L$ where 
the $K_\C$-stabilizer is not necessarily trivial. 
For that we first have to recall the concept of non-abelian 
cohomology (see \cite{S}, Sect. 5.1). 

\par Let $\Gamma$ and $H$ be groups. We assume that $\Gamma$ acts 
on $H$ by preserving the group law of $H$, in other words: 
there exists a homomorphism $\alpha:\Gamma \to \Aut(H)$. 
In the sequel we write for $g\in \Gamma$ and $h\in H$ 
$$ {}^{g}h:=\alpha(g)(h)\, .$$ 
By a cocycle of $\Gamma$ in $H$ we understand a map 
$$\theta:\Gamma\to H, \ \ g\mapsto \theta(g)$$
such that 
$$\theta(g_1 g_2)=\theta(g_1)\cdot  {}^{g_1} \theta(g_2)\,  .$$
The set of all cocycles is denoted by $Z^1(\Gamma, H)$. 
We call two cocycles $\theta, \theta'$ homologous if 
there exists an $h\in H$ such that 
$$\theta'(g)=h^{-1} \theta(g) {}^gh$$
for all $g\in \Gamma$. The corresponding set of 
equivalence  classes $H^1(\Gamma, H)= Z^1(\Gamma, H)/ \sim$ is refered 
to as the first cohomology set of $\Gamma$ with values in 
$H$. 

\par Henceforth we let $\Gamma=\mathrm{Gal}(\C|\R)$ be the Galois 
group of $\C|\R$. We write $\Gamma=\{ 1, \sigma\}$ with $\sigma$ the 
non-trivial element. Note  that $\Gamma$ acts on $K_\C$ by complex 
conjugation. In fact $\sigma$ induces the Cartan involution on $K_\C$.
Likewise $\Gamma$ acts on the stabilizer $S<K_\C$ 
of $L$.

\par Let us denote by $\O(\R)$ the real points of $\O$
and by $[\O(\R)]:= \O(\R)/ K$, the set of all $K$-orbits. 
Then for $k\in K_\C$ such that $z=k\cdot L \in \O(\R)$ we define a cocycle 
$$\theta_k(\sigma):= \sigma(k)^{-1} k\, . $$
Replacing $k$ by $ks$ for $s\in S$
results in $\theta_{ks}(\sigma)= \sigma(s)^{-1} \sigma(k)^{-1}k s$
which is homologous to $\theta_k$. 
Hence  the prescription $\theta_z:=\theta_k$ 
gives us a well defined element in $H^1(\Gamma, S)$. 
Further note that $\theta_{kz}=\theta_{z}$ for all $k\in K$. Therefore 
we obtain a map 

$$\Phi: [\O(\R)]\to H^1(\Gamma, S), \ \ [z]\mapsto \theta_z\, .$$

It is easy to check that $\Phi$ is injective and it remains to 
characterize the image.  
For that we consider the natural map between pointed sets
$$\Psi: H^1(\Gamma, S) \to H^1(\Gamma, K_\C)\, .$$ 
Define $\ker \Psi:=\Psi^{-1}({\bf 1})$. We recall that twisting 
(cf. \cite{S}, Sect. 5.4) implies that all pre-images 
of $\Psi$ are in fact kernels with $S$ and $K_\C$ replaced by twists.

\par We claim that $\Im \Phi =\ker \Psi$. The inclusion 
``$\subset$'' is obvious. Suppose that $\Psi(\theta)={\bf 1}$ is the trivial 
cocycle. Then $\theta(\sigma)=\sigma(k)^{-1} k $ for some $k\in K_\C$. 
As $\theta(\sigma)\in S$ it follows that $k\cdot L\in \O(\R)$ and our
claim is established. 

We have thus shown that 

\begin{equation} [\O(\R)]\simeq \ker \left(H^1(\Gamma, S)
\to H^1(\Gamma, K_\C) \right)\end{equation}
and thus 
\begin{equation} K\cdot L =\pf \quad \iff  \quad \ker \left(H^1(\Gamma, S)
\to H^1(\Gamma, K_\C) \right)=\{\mathbf 1\}\,  .\end{equation}
In the next step we wish to characterize the cohomology sets
involved. For $H^1(\Gamma, K_\C)$ we can use Harder's Theorem 
(cf. \cite{G}, Theorem III) to obtain 
\begin{equation} H^1 (\Gamma, K_\C)=\Z/2 \Z\, .\end{equation}
In order to discuss the structure of 
$H^1(\Gamma, S)$ we mention its more convenient description 
as a subset of $S$: 
$$H^1(\Gamma, S)=\{ s\in S\mid s\sigma(s)=1\}/\sim$$
where $s\sim s'$ if $s'= hs\sigma(h)^{-1}$ for some $h\in S$. 
Now the fact that $L$ is generic and $S$ is $\sigma$-stable implies 
that $S\subset K$ (see our argument from above).
Hence 
$$H^1(\Gamma, S)=\{ s\in S\mid s^2 =1\}/\sim$$
with $s\sim s'$ if $s'= hs h^{-1}$ for some $h\in S$.
\par To see an example let us consider $L$'s which correspond 
to subspaces $\Y'\backslash \Y''$. For those $L$ the stabilizer is
contained in $N_0\simeq (\Z/ 4\Z)\rtimes (\Z/2Z)$. 
Actually the stabilizer is either $\{{\bf 1}, m_2 s_0\}$
or  $\{{\bf 1}, s_0 m_2\}$.
Let us consider the first case: with $\gamma:= m_2 s_0$ 

$$\gamma=\begin{pmatrix} 0 & 1 & 0 \cr 1 & 0 & 0 \cr 
0 & 0 & -1\end{pmatrix}\, .$$
We have to show that $\gamma\not \in \ker \Psi$.  Now $\gamma\in \ker \Psi$ 
means $\gamma= k \sigma(k)^{-1}$ for some $k\in K_\C$, or equivalently 
$$\gamma \sigma(k) = k\, .$$ 
This means that the last row of $k$ consists of imaginary 
elements; a contradiction to the fact that the sum of their
squares adds up to one. 
Hence $\ker \Psi$ is trivial and therefore $K\cdot L= \pf$.

\end{document}